\documentclass{article}

\usepackage{amsfonts}
\usepackage{amssymb}
\usepackage[centertags]{amsmath}
\usepackage{amsfonts}
\usepackage{amssymb}
\usepackage{amsthm}
\usepackage{newlfont}

\newtheorem{theorem}{Theorem}

\newtheorem{lemma}{Lemma}

\newtheorem{corollary}{Corollary}

\theoremstyle{definition}
\newtheorem{definition}{Definition}

\newcommand{\inn}{\mathrm{ int}\,}
\newfont{\gothic}{eufm9 scaled 1200}
\begin{document}
\title{$D$--measurability and $t$-Wright convex functions}
\author{Eliza Jab{\l}o\'{n}ska}
\maketitle

\begin{abstract}
In the paper we will prove that each $t$-Wright convex function, which is bounded above on a $D$-measurable non-Haar meager set is continuous. Our paper refers to papers \cite{Olbrys}, \cite{Jablonska} and a problem posed by K.~Baron and R.~Ger.
\end{abstract}

\section{Introduction}
In 1972\footnote{{\em AMS Mathematics Subject Classification (2010):} 26A51, 39B62.\\{\em Keywords and phrases: t-Wright convex function, non-Haar meager set, D-meausrable set, D-measurable function}} J.P.R.~Christensen defined \textit{Haar null} sets in an abelian Polish group (which are also called Christensen zero sets) in such a way that in a locally compact group it is equivalent to the notion of Haar measure zero sets. P.~Fischer and Z. S{\l}odkowski \cite{FS} used his idea to define Christensen measurable sets and Christensen measurable functions.

It is well known (see \cite[Theorem 2]{Ch}, \cite[Example 6.2]{Mat}) that $0\in\inn (A-A)$ for each Christensen measurable nonzero set $A$ in an abelian Polish group $X$, as well as there is a Christensen measurable nonzero set $A\subset X$ such that $\inn (A+A)=\emptyset$. In fact it means that in an abelian Polish group the Steinhaus theorem holds, but a generalization of the Steinhaus theorem doesn't hold.

In 2013 U.B.~Darji introduced another family of "small" sets in an abelian Polish group $X$, which is equivalent to the notion of meager sets in a locally compact group; he called a set $A\subset X$ \textit{Haar meager} if there is a Borel set $B\subset X$ with $A\subset B$, a compact metric space $K$ and a continuous function $f:K\to X$ such that $f^{-1}(B+x)$ is meager in $K$ for all $x\in X$. He also proved that the family of all Haar meager sets is a $\sigma$-ideal.

In \cite{Jablonska} we introduced the definition of $D$-measurable sets and showed that the family of all $D$-measurable sets is a $\sigma$-algebra. Moreover, we proved that this family has two important properties analogous two the family of Christensen measurable sets:
\begin{itemize}
\item $0\in\inn (A-A)$ for each $D$-measurable non-Haar meager set $A\subset X$,
\item there is a $D$-measurable non-Haar meager set $A\subset X$ such that $\inn (A+A)=\emptyset$,
\end{itemize}
i.e. the Piccard theorem holds and a generalized Piccard's theorem doesn't hold in an abelian Polish group.

That is why we can say that the notation of $D$-measurability is a topological analog of Christensen measurability in an abelian Polish group (not necessary locally compact).

During the 21st International Symposium on Functional Equations (1983, Konolfingen, Switzerland) K. Baron and R. Ger posed the following problem:

 \textit{Does each Jensen convex function bounded above on a nonzero Christensen measurable set have to be continuous?}

 In 2013 A.~Olbry\'{s} \cite{Olbrys} answered the question. More precisely, he proved that for a nonempty convex open subset $D$ of a real linear Polish space each $t$-Wright convex function $f:D\to\mathbb{R}$, i.e. function satisfying the following condition
 $$
 \bigwedge_{x,y\in D} f(tx+(1-t)y)+f((1-t)x+ty)\leq f(x)+f(y)
 $$
 with a fixed number $t\in(0,1)$, bounded above on a nonzero Christensen measurable set $T\subset D$ has to be continuous.
Consequently, taking $t=\frac{1}{2}$ we obtain this result for Jensen convex functions.

 Here we will show that this theorem has its topological analog; i.e. each $t$-Wright convex function $f:D\to\mathbb{R}$ bounded above on a non-Haar meager $D$-measurable set $T\subset D$ is continuous.

 Let us mention that there are also known another conditions implying the continuity and convexity of $t$-Wright convex functions such as
lower semicontinuity \cite{Olbrys2},
continuity at one point \cite{Kominek}, \cite{Olbrys2},
Baire or Lebesgue masurability \cite{Olbrys1}-\cite{Olbrys2}.

\section{The main result}

Let $X$ be a real linear Polish space. First let us recall some necessary definitions and theorems on $D$-measurability from \cite{Darji} and \cite{Jablonska}.

\begin{definition}\cite[Definition~2.1]{Darji}, \cite[Definition~1]{Jablonska}
 A set $A\subset X$ is \textit{Haar meager} if there is a Borel set $B\subset X$ with $A\subset B$, a compact metric space $K$ and a continuous function
 $f:K\to X$ such that $f^{-1}(B+x)$ is meager in $K$ for all $x\in X$.
A set $A\subset X$ is  \textit{$D$-measurable} if $A=B\cup M$ for a Haar meager set $M\subset X$ and a Borel set
$B\subset X$.
\end{definition}

\begin{definition}\cite[Definition~2]{Jablonska}
 Let $X$ be an abelian Polish group and $Y$ be a topological space. A~mapping $f:X\to Y$ is a \textit{$D$-measurable function} if $f^{-1}(U)\in\mathcal{D}$ in $X$ for each open set $U\subset Y$.
\end{definition}

\begin{theorem}{\em \cite[Theorem~2.9]{Darji}}\label{t1}
The family $\mathcal{HM}$ of all Haar meager sets is a $\sigma$--ideal.
\end{theorem}

\begin{theorem}{\em \cite[Theorem~6]{Jablonska}}\label{t2}
The family $\mathcal{D}$ of all $D$-measurable sets is a $\sigma$--algebra.
\end{theorem}

\begin{theorem}{\em \cite[Theorem~2]{Jablonska}}\label{A-A}
If $A\in \mathcal{D}\setminus \mathcal{HM}$, then $0\in\inn (A-A).$
\end{theorem}

To prove the announced theorem we need two lemmas, which are analogous to Lemma~5 and Lemma~6 from \cite{Olbrys}.

\begin{lemma}\label{lem5}
Let $A\in \mathcal{D}\setminus \mathcal{HM}$ and $x\in X\setminus\{0\}$. Then there exist a Borel set $B\subset A$ such that the set
$k_x^{-1}(B+z)$ is non-meager with the Baire property in $\mathbb{R}$ for each $z\in X$, where $k_x:\mathbb{R}\to X$ is given by $k_x(\alpha)=\alpha x$.
\end{lemma}

\begin{proof}
Since $A\in \mathcal{D}\setminus \mathcal{HM}$, $A=B\cup M$, where $B$ is a Borel non-Haar meager set and $M$ is Haar meager. Since $B$ is not Haar meager, for every compact metric space $K$ and continuous function $g:K\to X$ there is a $y\in X$ such that $g^{-1}(B+y)$ is non-meager in $K$. In particular, let $K=[a,b]$ be a closed interval in $\mathbb{R}$ and $g_x:=k_x|_K$. Then there is a $y_x\in X$ such that the set $g_x^{-1}(B+y_x)$ is not meager in $[a,b]$ and, consequently, $g_x^{-1}(B+y_x)$ is not meager in $\mathbb{R}$, too.
 %(if it was a meager set in $\mathbb{R}$, it would be a countable sum of nowhere dense sets in $\mathbb{R}$, and hence a countable sum of nowhere %dense sets in $[a,b]$ with the induced topology; a contradiction).
Moreover,
$$
k_x^{-1}(B+y_x)\supset k_x^{-1}(B+y_x)\cap [a,b]=g_x^{-1}(B+y_x),
$$
so the set $k_x^{-1}(B+y_x)$ is not meager in $\mathbb{R}$.
Clearly, $k_x$ is linear, one-to-one and continuous. Hence, for each $z\in X$,
$$
k_x^{-1}(B+y_x)+k_x^{-1}(z-y_x)\subset k_x^{-1}(B+z).
$$
Thus the set $k_x^{-1}(B+z)$ is not meager in $\mathbb{R}$. Since $B$ is a Borel set, so does $k_x^{-1}(B+z)$. Whence $k_x^{-1}(B+z)$ has the Baire property, what ends the proof.
\end{proof}

To prove the second lemma we need two definitions from \cite{FS}:

\begin{definition}\cite[Definition~4]{FS}
A set $V\subset X$ is said to be \textit{midpoint convex} if $$\frac{v_1+v_2}{2}\in V\;\;\mbox{for every}\;\;v_1,v_2\in V.$$
\end{definition}

\begin{definition}\cite[Definition~5]{FS}
Let $V\subset X$ and denote
$$
\mathbb{Q}_2=\{\frac{n}{2^m}:\;n,m\in\mathbb{Z}\}.
$$
An $x_0\in V$ is said to be a \textit{$\mathbb{Q}_2$-internal point of $V$} if for every $x\in X$ there exists an $\varepsilon_x>0$ such that $x_0+\rho x\in V$ for each $\rho\in\mathbb{Q}_2\cap (-\varepsilon,\varepsilon).$
\end{definition}

\begin{lemma}\label{lem6}
Let $D\subset X$ be a nonempty convex open set. Let $V\subset D$ be $D$-measurable midpoint convex set. Then the set of all $\mathbb{Q}_2$-internal points of $V$ is open.
\end{lemma}

\begin{proof}
Let $x_0$ be a $\mathbb{Q}_2$-internal point of $V$ and $U:=(V-x_0)\cap (-V+x_0)$. Clearly, $U$ is symmetric with respect to the origin, $D$-measurable, midpoint convex and the origin is a $\mathbb{Q}_2$-internal point of $U$. Hence $\bigcup_{n=1}^{\infty} 2^nU=X$ and thus, by Theorem~\ref{t1}, the set $2^kU$ is not Haar meager for some $k\in\mathbb{N}$. This implies that $U$ is not Haar meager, too.

Indeed, define $\phi:X\to X$ as follows:
$$
\phi(x)=2^k x\;\;\mbox{for}\;\;x\in X.
$$
We will show that $\phi(M)$ is Haar meager for each Haar meager set $M \in X$.

So, for Haar meager $M$ there exist a compact metric space $K_0$ and a continuous function $f_0:K_0\to X$ such that $f_0^{-1}(x+M)$ is meager in $K_0$ for each $x\in X$. Define $g_0:=\phi\circ f_0$. Then, for each $y\in X$, we have
$$
g_0^{-1}(y+\phi(M))=g_0^{-1}\left(2^k\left(\frac{y}{2^k}+M\right)\right)=f_0^{-1}\left(\phi^{-1}\left(2^k\left(\frac{y}{2^k}+M\right)\right)\right)=
f_0^{-1}\left(\frac{y}{2^k}+M\right).
$$
It means that the set $g_0^{-1}(y+\phi(M))$ is meager in $K_0$ and, consequently, $\phi(M)$ is Haar meager.

In this way we obtain that $U$ is not Haar meager. But $U$ is $D$-measurable, so $U=B\cup M$, where $B$ is a Borel non-Haar meager set and $M$ is Haar meager. Thus, in view of Theorem~\ref{A-A},
$0\in\inn (U-U).$
But $U$ is symmetric with respect to the origin, so $U-U=2U$ and hence $U$ is a neighborhod of the origin. Moreover, $x_0+U\subset V$, what ends the proof.
\end{proof}

Now, we can prove the main result.

\begin{theorem}
Let $D\subset X$ be a nonempty convex open set.
Each $t$-Wright convex function $f:D\to\mathbb{R}$ bounded above on a non-Haar meager $D$-measurable set $T\subset D$ is continuous.
\end{theorem}

In view of Theorem~\ref{t1} and Theorem~\ref{t2}, the proof of the above theorem runs in the same way as the proof of Theorem~8 in \cite{Olbrys}; it is enough to replace\\
\cite[Lemma~5]{Olbrys} by Lemma~\ref{lem5},\\
 \cite[Lemma~6]{Olbrys} by Lemma~\ref{lem6},\\
 \cite[Theorem~5]{Olbrys} by \cite[Theorem~4]{Olbrys}.

\begin{corollary}
Let $D\subset X$ be a nonempty convex open set. Each Jensen convex function $f:D\to\mathbb{R}$ bounded above on a non-Haar meager $D$-measurable set $T\subset D$ is continuous.
\end{corollary}

Next, in the same way as Theorem 10 in \cite{Olbrys}, we obtain

\begin{theorem}
Let $D\subset X$ be a nonempty convex open set.
Let $f:D\to\mathbb{R}$ be a $t$-Wright convex function. If there exists a $D$-measurable function $g:D\to\mathbb{R}$ such that
$$
\bigwedge_{x\in D}\; f(x)\leq g(x),
$$
then $f$ is continuous.
\end{theorem}

\noindent
Department of Mathematics\\
Rzesz{\'o}w University of Technology\\
Powsta\'{n}c\'{o}w Warszawy 12\\
35-959 Rzesz{\'o}w\\
POLAND

\noindent {\em e-mail}: elizapie@prz.edu.pl
\end{document}